\documentclass[1p]{elsarticle}
\usepackage{amsmath, amssymb,amsthm,stmaryrd}
\usepackage{graphicx,mdwtab}

\newtheorem{theorem}{Theorem}

\newtheorem{lemma}[theorem]{Lemma}

\newtheorem{proposition}[theorem]{Proposition}
\newtheorem{corollary}[theorem]{Corollary}
\newtheorem{problem}{Problem}
\theoremstyle{definition}

\theoremstyle{remark}

\newtheorem{conjecture}{Conjecture}

\newcommand{\shm}{\,\triangledown\,}
\newcommand{\shtm}{\,\widetilde{\triangledown}\,}

\newcommand{\mtrdens}[1]{\,\widetilde{\nabla}\kern-2.8pt\raisebox{1.6pt}{$\scriptstyle /$}_{#1}}
\newcommand{\bbbn}{\mathbb{N}}

 \newcommand{\GRA}{{\mathcal Graph}}
 
 \newlength{\graphshift}
 \newcommand{\nlongrightarrow}{\relbar\!\joinrel\not\relbar\joinrel\!\!\rightarrow}
\newcommand{\duality}[3][1cm]{\setlength{\graphshift}{#1}\raisebox{-.4\graphshift}{\includegraphics[height=\graphshift]{#2}}\quad\nlongrightarrow\quad
G\qquad\iff\qquad  G\quad\longrightarrow \quad\raisebox{-.4\graphshift}{\includegraphics[height=\graphshift]{#3}}}

\begin{document}
\begin{frontmatter}
\journal{special issue}
\title{On First-Order Definable Colorings}
\author{J.~Ne\v set\v ril\fnref{fn1}}
\address{Computer Science Institute of Charles University (IUUK)\\
 and Institute of Theoretical Computer Science (ITI)\\
  Malostransk\' e n\' am.25, 11800 Praha 1, Czech Republic}
 \ead{nesetril@iuuk.mff.cuni.cz}
\fntext[fn1]{Supported by grant ERCCZ LL-1201, 
  CE-ITI P202/12/6061,  and by the European Associated Laboratory ``Structures in
Combinatorics'' (LEA STRUCO)}
\author{P.~Ossona de Mendez\corref{cor1}\fnref{fn2}}
\address{
Centre d'Analyse et de Math\'ematiques Sociales (CNRS, UMR 8557)\\
  190-198 avenue de France, 75013 Paris, France\\
  	--- and ---\\
Computer Science Institute of Charles University (IUUK)\\
   Malostransk\' e n\' am.25, 11800 Praha 1, Czech Republic
  }
 \ead{pom@ehess.fr}
\cortext[cor1]{Corresponding author}
\fntext[fn2]{Supported by grant ERCCZ LL-1201,  by the European Associated Laboratory ``Structures in
Combinatorics'' (LEA STRUCO), 
and partially supported by ANR project Stint under reference ANR-13-BS02-0007}

\begin{abstract}
We address the problem of characterizing $H$-coloring problems that are first-order definable on a fixed class of relational structures. In this context, we give also several characterizations of a
homomorphism dualities arising in a class of structure.
\end{abstract}

\begin{keyword}
graph \sep coloring \sep first-order logic \sep bounded expansion class
\MSC{05C15} {coloring of graphs and hypergraphs}
\end{keyword}
\end{frontmatter}
\section{Introduction}
Recall that {\em classical model theory} studies properties of abstract
mathematical structures (finite or not) expressible in first-order logic \cite{Hodges1993}, and
{\em finite model theory} is the study of first-order logic (and its
various extensions) on finite structures \cite{Ebbinghaus1996},
\cite{Libkin2004}.

{\em Constraint Satisfaction Problems} (CSPs), and more specifically $H$-coloring problems, are standard examples of 
problems which can be expressed in monadic second order logic but usually not in the first-order logic. Of course, expressing a $H$-coloring problem in first-order logic
would be highly appreciable, as it would allow fast checking (in at most polynomial time) although problems expressed in monadic second order logic are usually NP-complete.
In this direction, it has been proved by Hell and Ne{\v s}et{\v r}il
\cite{hellne90} that in the context of finite undirected graphs the $H$-coloring
problem is NP-complete unless $H$ is bipartite, in which case the $H$-coloring
problem is clearly polynomially solvable. This, and a similar dichotomy result of Schaefer~\cite{Schaefer1978},
led Feder and Vardi \cite{Feder1993,0914.68075} to 
formulate the celebrated Dichotomy Conjecture which asserts that, for every
constraint language over an arbitrary finite domain, the corresponding constraint
satisfaction problems are either solvable in polynomial time, or are NP-complete. It was soon noticed
that this conjecture is equivalent to the existence of a dichotomy for (general)
$H$-coloring problems, and in fact it suffices to prove it for oriented graphs (see \cite{0914.68075} and \cite{HN}).

Alternatively, the class P of all polynomially solvable problems can be described as the class of problems expressible (on ordered structures) in first-order logic with a least fixed point operator \cite{Vardi1982,Immerman1982}.
On the other hand the class NP may be characterized (up to polynomial equivalence)  as the class of all problems which have a lift (or expansion) determined by forbidden homomorphisms from a finite set \cite{KN}.
Hence, we are led naturally to the question of descriptive complexity of classes of structures corresponding to $H$-coloring problems. A particular case is the question whether a $H$-coloring problem
 may be expressed in first-order logic or not.

In this paper, we will consider the relativized version of the problem of
first-order definability of $H$-coloring problems to graphs (or structures)
belonging to a fixed class $\mathcal C$:
\begin{problem}
\label{pb:1}
Given a fixed class $\mathcal C$ of graphs (directed graphs, relational
structures), determine which $H$-coloring problems are first-order
definable in $\mathcal C$. Explicitly,  determine for which graphs (directed graphs,
relational structures) $H$  there exists a first-order sentence
$\phi_H$ such that
$$
\forall G\in\mathcal C:\qquad (G\models\phi_H)\iff(G\rightarrow H).
$$
\end{problem}
The case where $\mathcal C$ is the whole class of all finite graphs (all
 finite directed graphs, all finite relational structures with given finite
 signature) is well understood. Atserias~\cite{Atserias2005a,Atserias2008} 
 and Rossman~\cite{Rossman2007}
 proved that in this case first-order definable 
 $H$-colorings correspond exactly to {\em finite homomorphism dualities}, and
 these dualities have been fully characterized (for
 undirected graphs, by Ne{\v s}et{\v r}il and Pultr \cite{NPULTR};
 for directed graphs, by Kom{\' a}rek \cite{Kom'arek1988}; 
 for general finite structures, by Ne{\v s}et{\v r}il and Tarif \cite{NT}) as follows:
 
 \begin{theorem}[\cite{NT}]
For any signature $\sigma$ and any finite set $\mathcal F$ of
$\sigma$-structures the following two statements are equivalent:
\begin{enumerate}
  \item There exists $D$ such that $\mathcal F$ and $D$ form a finite duality,
  that is:
  $$\forall \text{ finite }G:\qquad (\forall F\in\mathcal F,\ F\nrightarrow
  G)\quad\iff\quad (G\rightarrow D)$$
  
\item $\mathcal F$ is homomorphically equivalent to a set of finite (relational)
trees.
\end{enumerate}
 \end{theorem}
Note that an example of such a duality for the class of all finite directed graphs
 is the Gallai-Hasse-Roy-Vitaver theorem
 \cite{Gallai1968,Hasse1964,Roy1967,Vitaver1962}, which states that for every
 directed graph $\vec{G}$ it holds:
\begin{equation*}
\vec{P}_{k+1}\nrightarrow\vec{G}\qquad\iff\qquad \vec{G}\rightarrow\vec{T}_k.
\end{equation*}

For general classes of graphs the answer is more complicated . 
For instance, let $\mathcal C$ be the class of (undirected) toroidal graphs
and let $\phi$ be the sentence 

$$
\forall x_0\,\forall x_1\,\dots\forall x_{10}\quad\bigvee_{i=0}^{10}
\neg(x_i\sim x_{i+1})\vee \neg(x_i\sim x_{i+2})\vee \neg(x_i\sim x_{i+3}),$$

where additions are considered modulo $11$ and where $u\sim v$ denotes
 that $u$ and $v$ are adjacent. Then, it follows from \cite{Thomassen199411}
 that a graph $G\in\mathcal C$ satisfies $\phi$ if and only if it is
 $5$-colourable. This property can be alternatively be expressed by the
 following {\em restricted duality\/}:
  
 $$
 \forall G\in\mathcal C:\qquad
 \duality[15mm]{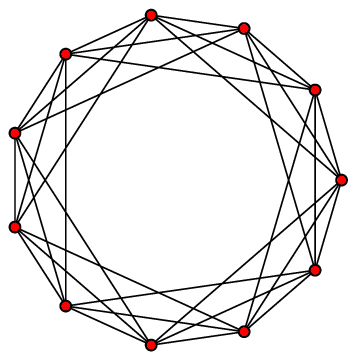}{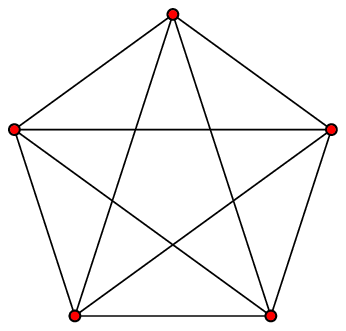}
 $$
 
In fact the class of toroidal graphs has all restricted dualities in the following sense: for every connected $F$ there exists $H_F$ such that $F\nrightarrow H_F$ and for every toroidal graph $G$ holds
$$F\nlongrightarrow G\qquad\iff\qquad G\longrightarrow H_F.$$ 
 
For a general class of graphs $\mathcal C$, Problem~\ref{pb:1} is very complex. We have to specialize.
Hence we first require that the studied class $\mathcal C$ has some
basic properties: we assume that
\begin{itemize}
\item $\mathcal C$ is {\em hereditary} (meaning that every induced subgraph of a graph in $\mathcal C$ is in $\mathcal C$);
\item $\mathcal C$ is {\em addable} (meaning that disjoint unions of graphs in $\mathcal C$ are in $\mathcal C$);
\item $\mathcal C$ is {\em topologically closed} (meaning that every subdivision of a graph in $\mathcal C$ is in $\mathcal C$).
\end{itemize}
We approach the problem of characterizing first-order definable colorings by 
first discriminating between the cases of sparse and dense classes of graphs 
using our {\em class taxonomy}~\cite{ECM2009,Nevsetvril2010a,Sparsity}, which seems relevant here.
 The second central ingredient of our study is the notion of {\em homomorphism preservation theorem} (HPT) for a class $\mathcal C$, which was investigated in \cite{Atserias2006,Dawar,Dawar2010, Sparsity,Rossman2007}.
 Our approach can be outlined as follows:

If there exists
some integer $p$ such that every $p$-subdivision of a complete graph appears as a subgraph of some graph in $\mathcal C$ (meaning that $\mathcal C$ is {\em somewhere dense}) then we prove that every first-order definable $H$-coloring defines a restricted duality on the subclass $\mathcal C'\subseteq\mathcal C$ of the $p$-subdivisions of simple graphs (follows from HPT for $\mathcal C'$). Using a classical construction of Erd\H os~\cite{ErdH1959}, we deduce that if a $H$-coloring problem is first-order definable on $\mathcal C$ then either $H$ is bipartite, or the odd-girth of $H$ is at most $2p+1$.

Otherwise (meaning that $\mathcal C$ is {\em nowhere dense}) it follows from HPT for nowhere dense classes that
every first-order definable $H$-coloring defines a restricted duality on $\mathcal C$.
 In the case where
there exists
some integer $p$ such that $p$-subdivisions of graphs with arbitrarily large average degree appear as subgraphs of graphs in $\mathcal C$ (meaning that $\mathcal C$ is not a {\em bounded expansion} class)  we prove that $H$ cannot be a restricted dual of a non-bipartite graph with arbitrarily large odd-girth. Modulo some reasonable conjecture (Conjecture~\ref{conj:nd}), we get that (as in the somewhere dense case) $H$ is either bipartite or has bounded odd-girth.

In the reminding case (where $\mathcal C$ is a bounded expansion class), 
we prove that first-order definable coloring correspond to restricted dualities, hence are exactly defined by a sentence expressing that no homomorphism exists from one of the connected graphs belonging to a finite set.

This study naturally leads to the following conjecture:
\begin{conjecture}
\label{conj:main}
Let $\mathcal C$ be a hereditary addable topologically closed class of graphs.
The following properties are equivalent:
\begin{enumerate}
\item   for every integer $p$ there is a non-bipartite graph $H_p$ of odd-girth strictly greater than $2p+1$ and a first order definable class $\mathcal D_p$ such that
a graph $G\in\mathcal{C}$ is $H_p$-colorable if and only if $G\in\mathcal D_p$.
Explicitly, there exists a
formula $\Phi_p$ such that for every graph $G\in\mathcal C$ holds
$$(G\vDash \Phi_p)\quad\iff\quad (G\rightarrow H_p);$$
\item the class $\mathcal C$ has bounded expansion.
\end{enumerate}
\end{conjecture}

In this paper, we make a significant progress toward a solution to Conjecture~\ref{conj:main}. We state a structural conjecture, Conjecture~\ref{conj:nd}, about nowhere dense classes that fail to be bounded expansion classes, which expresses that such classes are characterized by the existence (as subgraphs of graphs in the class) of $p$-subdivisions of graphs with arbitrarily large chromatic number and girth. More precisely:

\begin{conjecture}
\label{conj:nd}
Let $\mathcal C$ be a monotone nowhere dense class that does not have bounded expansion.
Then there exists an integer $p$ such that $\mathcal C$ includes $p$-subdivisions of
graphs with arbitrarily large chromatic number and girth.
\end{conjecture}

Our main result toward a characterization of first-order definable colorings is the following reduction.

\begin{theorem}
\label{thm:main}
Let $\mathcal C$ be a hereditary topologically closed class of graphs that is somewhere dense or has bounded expansion. Then
the following properties are equivalent:
\begin{enumerate}
\item   for every integer $p$ there is a non-bipartite graph $H_p$ of odd-girth strictly greater than $2p+1$ such that $H_p$-coloring is first-order definable on $\mathcal C$;
\item the class $\mathcal C$ has bounded expansion.
\end{enumerate}

Moreover, if Conjecture~\ref{conj:nd} holds then the statement holds for every hereditary addable topologically closed class of graphs, that is:
if Conjecture~\ref{conj:nd} holds then Conjecture~\ref{conj:main} also holds.
\end{theorem}

In support to Conjecture~\ref{conj:nd}, we prove (Proposition~\ref{prop:conj})
 that it would follow
from 
a positive solution to any of the following
 two well known conjectures.

\begin{conjecture}[Erd\H os and Hajnal \cite{Erdoshyp}]
\label{conj:EH}
For every integers $g$ and $n$ there exists an integer $N=f(g,n)$ such that every graph $G$ with chromatic number at least $N$ has a subgraph $H$ with girth at least $g$ 
and chromatic number at least $n$.
\end{conjecture}

The case $g = 4$ of the conjecture was proved by
R\"odl \cite{Rodl1977}, while the general case is still open.
Remark that the existence of graphs of both arbitrarily high chromatic number and high girth
is a well known result of Erd\"os \cite{ErdH1959}.

\begin{conjecture}[Thomassen \cite{Thomassen1983}]
\label{conj:Thomassen}
For all integers $c, g$ there exists an integer $f(c, g)$
such that every graph $G$ of average degree at least $f(c, g)$ contains a subgraph
of average degree at least $c$ and girth at least $g$.
\end{conjecture} 

The case $g=4$ of this conjecture is a direct consequence of the simple fact
that every graph can be made bipartite by deleting at most half of its edges.
 The case $g=6$ has been proved in \cite{K`uhn2004}. 


Although we do not settle Conjecture~\ref{conj:main}, our study led us to the two following characterization theorems of classes that have all restricted dualities.

\begin{theorem}
\label{thm:sub}
Let $\mathcal C$ be a topologically closed class of  graphs. 

The following properties are equivalent:
\begin{enumerate}
	\item the class $\mathcal C$ has bounded expansion;
	\item the class $\mathcal C$ has all restricted dualities;
	\item for every odd integer $g$ there exists a graph $H_g$ with odd-girth greater than $g$ such that every graph $G\in\mathcal C$ with odd-girth greater than $g$ has a homomorphism to $H_g$.
	\end{enumerate}
\end{theorem}

These results motivate more general studies of classes of relational structures having all restricted dualities, which includes classes of structures whose Gaifman graphs form a class with bounded expansion. In a very general setting, we obtain the following characterization:

\begin{theorem}
\label{thm:approx}
Let $\mathcal C$ be a class of $\sigma$-structures. Then $\mathcal C$
is bounded and has all restricted dualities  if and only if  for
every integer $t$ there is an integer $N(t)$ such that for every $\mathbf A\in\mathcal C$ there exists a $\sigma$-structure $\mathbf A_t$ (called
{\em $t$-approximation} of $\mathbf A$) such that
\begin{itemize}
 \item $\mathbf A_t$ has order $|A_t|$ at most $N(t)$,
  \item $\mathbf A\rightarrow \mathbf A_t$,
  \item every substructure $\mathbf F$ of $\mathbf A_t$ with order $|F|<t$ has a homomorphism to  $\mathbf A$.
\end{itemize}
\end{theorem}

This paper is organized as follows:

In Section~\ref{sec:prelim}, we recall the notions needed in the development  of our study, in particular class taxonomy and basics on relational structures.

In Section~\ref{sec:hpt}, we discuss classes satisfying a homomorphism preservation theorem (HPT). On these classes, first-order definable $H$-colorings correspond to finite restricted dualities. Moreover, when the considered class is addable, they correspond to finite restricted dualities with connected templates.

In Section~\ref{sec:ard}, we discuss classes having all restricted dualities, that is classes such that to every finite set $\mathcal F$ of connected templates correspond a graph $H$ such that a graph $G$ in the class is $H$-colorable if and only if none of the templates in $\mathcal F$ has a homomorphism to $G$.

\section{Taxonomy of Classes of Graphs}
\label{sec:prelim}
In the following, we denote by $\GRA$ the class of all finite graphs. 
A class of graphs $\mathcal C$ is {\em monotone} (resp. {\em hereditary}, {\em topologically closed}) if every subgraph (resp. every induced subgraph, every subdivision) of a graph in $\mathcal C$ also belongs to $\mathcal C$. Notice that if a class $\mathcal C$ is both hereditary and topologically closed it is also monotone: If $H$ is a subgraph of a graph $G\in\mathcal C$, the graph $H$ is an induced subgraph of the graph $G'\in\mathcal C$ obtained from $G$ by subdividing every edge not in $H$, hence $H\in\mathcal C$. For a graph $G$, we denote by $\omega(G)$ its clique number, by $\chi(G)$ its chromatic number, and by $\overline{\rm d}(G)$ the average degree of its vertices. By extension, for a class of graphs $\mathcal{C}$ we define
\begin{align*}
\omega(\mathcal C)&=\sup\{\omega(G),\ G\in\mathcal{C}\}\\
\chi(\mathcal C)&=\sup\{\chi(G),\ G\in\mathcal{C}\}\\
\overline{\rm d}(\mathcal C)&=\sup\{\overline{\rm d}(G),\ G\in\mathcal{C}\}
\end{align*}

We proposed in \cite{ND_logic,ND_characterization,POMNI} a general classification scheme for graph classes which is based on the density of shallow (topological) minors (we refer the interested reader
to the monography \cite{Sparsity}).
This classification can be defined in several very different ways and we give here one of the simplest definitions, which relates to subdivision:

A {\em subdivision} (resp. a {\em $k$-subdivision}) of an edge $e=\{u,v\}$ of a graph $G$ consists in replacing the edge $e$ by a path (resp. by a path of length $k+1$) with endpoints $u$ and $v$.
A {\em subdivision} of a graph $G$ is a graph $H$ resulting from $G$ by 
subdividing edges; the graph $H$ is the {\em $k$-subdivision} of $G$ if it has been obtained from $G$ by $k$-subdividing all the edges (i.e. if all the edges of $G$ have been replaced by paths of length $k+1$). The graph $H$ is a 
{\em $\leq k$-subdivision} of $H$ if it has been obtained by replacing each edge of $G$ by a path of length at most $k+1$. 

Let $p$ be a half-integer. A graph $H$ is a {\em shallow topological minor} of a graph $G$ {\em at depth $p$} if some $\leq 2p$-subdivision of $H$ is a subgraph of $G$;
	the set of all shallow topological minors of $G$ at depth $p$ is denoted by $G\shtm p$ and, more generally, $\mathcal C\shtm p$ denotes the class of all shallow topological minors at depth $p$ of graphs in $\mathcal C$.

A class of undirected graphs $\mathcal C$ is {\em somewhere dense} if there exists an integer $p$ such that
 the $p$-th subdivision of every finite graph $H$ may be found as a subgraph of some graph in $\mathcal C$; it is  {\em nowhere dense} otherwise.
 
 In other words,  the class $\mathcal C$ is  nowhere dense if
$$\forall p\in\bbbn,\qquad \omega(\mathcal C\shtm p)<\infty.$$

A particular type of nowhere dense classes will be of particular importance in this paper:
A class $\mathcal C$ has {\em bounded expansion} \cite{POMNI} if
$$\forall p\in\bbbn,\qquad\overline{\rm d}(\mathcal C\shm p)<\infty.$$
Among the numerous equivalent characterizations that can be given for the property of having bounding expansion, we will make use of a characterization based on the chromatic numbers of the shallow topological minors of the graphs in the class. This characterization can be deduced from the following result of Dvo{\v r}{\'a}k \cite{Dvo2007,Dvov2007} (see also \cite{Sparsity}):
\begin{lemma}
\label{lem:dvo}
Let $c\geq 4$ be an integer and let $G$ be a graph with minimum degree $d>56(c-1)^2\frac{\log (c-1)}{\log c-\log (c-1)}$. Then the graph $G$ contains a subgraph $G'$ that is the $1$-subdivision of a graph with chromatic number $c$.
\qed
\end{lemma}
Hence the following characterization of classes having bounded expansion:
\begin{theorem}
\label{thm:BEchi}
A class $\mathcal C$ has bounded expansion if and only if it holds
\begin{equation}
\forall p\in\bbbn,\qquad\chi(\mathcal C\shtm p)<\infty.
\end{equation}
\end{theorem}
\begin{proof}
According to Lemma~\ref{lem:dvo}, for every graph $G$ and every integer $p$ there exists an integer $C$ such that:
\begin{equation*}
\overline{\rm d}(G\shtm p)\leq C\,\chi(G\shtm (2p+1/2))^4.
\end{equation*}
Moreover, as every graph $G$ is $(\lfloor \overline{\rm d}(G\shtm 0)\rfloor+1)$-colorable, every graph in $G\shtm p$
is $(\lfloor \overline{\rm d}(G\shtm p)\rfloor+1)$-colorable, that is:
\begin{equation*}
\overline{\rm d}(G\shtm p)\geq \chi(G\shtm p)-1.
\end{equation*}
The result follows from these two inequalities.
\end{proof}
Thus we see that  parameters $\overline{\rm d}$ and $\chi$  can be used to define bounded expansion classes, although nowhere dense classes are defined 
by means of the parameter $\omega$.

Characterizing nowhere dense classes that do not have bounded expansion in a
structural way is challenging, and thus we proposed Conjecture~\ref{conj:nd}, from which Conjecture~\ref{conj:main} would follow. 

 We now prove that Conjecture~\ref{conj:nd} follows from any of the Conjecture~\ref{conj:EH} (by Erd\H os and Hajnal) or Conjecture~\ref{conj:Thomassen} (by Thomassen):

\begin{proposition}
\label{prop:conj}
If either Conjecture~\ref{conj:EH}  or Conjecture~\ref{conj:Thomassen}
holds, then so does Conjecture~\ref{conj:nd}.
\end{proposition}
\begin{proof}
Let $\mathcal C$ be a nowhere dense class of graphs which is not a bounded expansion class. By definition of a bounded expansion class, there exists an integer $q$ such that $\mathcal C$ includes $\leq q$-subdivisions of graphs with arbitrarily large average degree hence (by standard pigeon-hole argument) there is an integer $p$ such that $\mathcal C$ includes exact $q$-subdivisions of graphs with arbitrarily large average degree. 

Assume Conjecture~\ref{conj:EH} holds. Define $p=2q+1$, and let $g,n\in\bbbn$.  According to the statement of the conjecture, there exists $N$ such that every graph with chromatic number at least $N$ has a subgraph with girth at least $g$ and chromatic number at least $n$. We can assume $N\geq 4$. 
Let $d\geq 56(N-1)^2\frac{\log (N-1)}{\log N-\log (N-1)}$. 
Let $G\in\mathcal C$ be such that $G$ includes the $q$-subdivision of a graph  with average degree at least $2d$ hence the $q$-subdivision of a graph $H$ with minimum degree at least $d$.
According to Lemma~\ref{lem:dvo}, $H$ has a subgraph $H'$ that is the $1$-subdivision  of a graph with chromatic number $N$. It follows that $G$ has a subgraph which is a $p$-subdivision of a graph $K$ with chromatic number $N$. 
According to Conjecture~\ref{conj:EH}, the graph $K$ has a subgraph $K'$ which has chromatic number at least $n$ and girth at least $g$. It follows that $G$ contains the $p$-subdivision of a graph with chromatic number at least $n$ and girth at least $g$. Thus Conjecture~\ref{conj:nd} holds.

Assume that Conjecture~\ref{conj:Thomassen} holds. Define $p=2q+1$, and  let $g,n\in\bbbn$.  Let $d\geq 56(n-1)^2\frac{\log (n-1)}{\log n-\log (n-1)}$. 
According to the statement of the conjecture, there exists $N$ such that every graph with average degree at least $N$ has a subgraph with girth at least $2g+1$ and average degree at least $2d$ hence a subgraph with girth at least $2g+1$ and minimum degree at least $d$.
Let $G\in\mathcal C$ be such that $G$ includes the $q$-subdivision of a graph  with average degree at least $N$. Then $G$ has a subgraph which is the $q$-subdivision of a graph $H$ with minimum degree at least $d$ and girth at least $2g+1$. According to Lemma~\ref{lem:dvo}, $H$ has a subgraph which is the $1$-subdivision of a graph with chromatic number at least $n$. This subgraph is a $p$-subdivision of a graph with girth at least $g$ and chromatic number at least $n$.  Thus Conjecture~\ref{conj:nd} holds.
\end{proof}

\section{Homomorphism Preservation Theorems}
\label{sec:hpt}
Suppose that an $H$-coloring problem is first-order definable. By this we mean that
there is a first-order sentence $\Phi$ such that
$$G\rightarrow H\qquad\iff\qquad G\models\Phi.$$

It immediately follows that $\neg \Phi$ is preserved by homomorphisms: 
$$ G\vDash\,\neg\Phi\quad\text{and}\quad G\rightarrow G'\qquad\Longrightarrow\qquad G'\vDash\,\neg\Phi$$
 (for otherwise $G\rightarrow G'\rightarrow H$ hence $G\vDash\,\Phi$, a contradiction).

 Such a property suggests that such a formula $\Phi$ could be equivalent to a
 formula with a specific syntactic form.
 Indeed the classical {\em Homomorphism Preservation Theorem} (HPT) asserts that a
 first-order formula is preserved under homomorphisms on all structures if, and
 only if, it is logically equivalent to an existential-positive formula. 
 The terms ``all structures'', which means finite and infinite structures, is
 crucial in the statement of these theorems.
 
 \subsection{Finite Structures}
 
 It was not known until recently whether HPT would hold when relativized
 to the finite. In fact other well known theorems relating
 preservation under some specified algebraic operation and certain syntactic
 forms, like \L o\'s-Tarski theorem or Lyndon's theorem, fail in the finite.

 However, the finite relativization of the homomorphism preservation has been
 proved to hold by B. Rossman \cite{Rossman2007} for general relational structures.

 \begin{theorem}[\cite{Rossman2007}]
 \label{thm:rossman}
 Let $\phi$ be a first order formula. Then,
 \begin{equation*}
\mathbf  G\rightarrow \mathbf H\text{ and }\mathbf G\vDash\phi\quad\Longrightarrow\quad \mathbf H\vDash\phi
 \end{equation*}
  holds for all finite relational structures $\mathbf G$ and $\mathbf H$ if and only if for finite relational structures $\phi$ is equivalent to an existential first-order formula.
 \end{theorem}
 It follows that for finite relational structures, the only $\mathbf H$-coloring problems which are expressible in first-order logic are those for which there exists a finite family $\mathcal F$ of finite structures with the property that for every structure $\mathbf G$ the following finite homomorphism duality holds:
 \begin{equation}
 \exists\mathbf F\in\mathcal F\quad \mathbf F\rightarrow\mathbf G\qquad\iff\qquad\mathbf G\nrightarrow\mathbf H.
 \end{equation}
 
In this paper, we will be mostly interested by graphs, although relational structures will be considered in Section~\ref{sec:ard}. Definitions and constructions concerning relational structures are particularly discussed in Section~\ref{sec:rel}.

\subsection{Nowhere dense classes}
If we want to relativize Theorem~\ref{thm:rossman}, we should consider
each relativization as a new problem. The {\L}o\'s-Tarski theorem, for instance, holds in general, yet fails when relativized to the finite, but holds when relativized to
hereditary classes of structures with bounded degree which are closed under disjoint union \cite{Atserias2005}. These examples stress again that some properties of structures (in general) and graphs (in particular) need, at times, to be studied in the context of a fixed class, in order to state a relativized version of a general statement which could fail in general.

In this context Atserias, Dawar and Kolaitis defined classes of graphs called {\em wide}, {\em almost wide} and {\em quasi-wide} (cf. \cite{Dawar2007} for instance).
It has been proved in \cite{Atserias2005} that the extension preservation theorem holds in any class $\mathcal C$ that is wide, hereditary (i.e. closed under taking substructures) and closed under disjoint unions, for instance  hereditary classes with bounded degree that are closed under disjoint unions.
Also, it has been proved in \cite{Atserias2004} \cite{Atserias2006} that the homomorphism preservation theorem holds in any class $\mathcal C$ that is almost wide, hereditary and closed under disjoint unions. Almost wide classes of graphs include classes of graphs which exclude a minor \cite{Kreidler1999}.
In~\cite{Dawar2010} Dawar proved that the homomorphism preservation theorem holds in any hereditary quasi-wide class that is closed under disjoint unions. 

\begin{theorem}[\cite{Dawar2010}]
\label{thm:hptnd}
Let $\mathcal C$ be a hereditary addable quasi-wide class of graphs. Then the homomorphism preservation theorem holds for $\mathcal C$.
\end{theorem}

Moreover,  we have proved that hereditary quasi-wide classes of graphs are exactly hereditary nowhere dense classes \cite{ND_logic}:

\begin{theorem}
\label{thm:qw}
A hereditary class of graphs $\mathcal C$ is quasi-wide if and only if it is nowhere dense.
\end{theorem} 

Thus it follows from Theorems~\ref{thm:hptnd} and~\ref{thm:qw}
that the relativization of the homomorphism preservation theorem holds for every
hereditary addable nowhere dense class of graphs.
But nowhere dense classes are not the only classes with relativized homomorphism preservation theorem.
In the next section we show HPT also holds for some nowhere dense classes.

\subsection{Somewhere dense classes} 
We now show that relativized homomorphism preservation theorems are preserved by particular interpretations, from which will deduce that relativized homomorphism preservation theorems hold
 for the classes ${\rm Sub}_{q}(\GRA)$ of all $q$-subdivisions of (simple) finite graphs.
This is of particular interest as somewhere dense classes (i.e. classes which fail to be nowhere dense) are characterized by containment of classes ${\rm Sub}_{q}(\GRA)$ for some $q$.
 
In the framework of the model theoretical notion of {\em interpretation} (see, for instance \cite[pp.~178-180]{Lascar2009}), we can construct the $q$-subdivision $\mathsf I(G)$ of a graph $G$ by means of first-order formulas on the $q$-tuples of vertices of $G$:
\begin{itemize}
\item vertices of $\mathsf I(G)$ are the equivalence classes $x$ of the $(q+1)$-tuples $(v_1,\dots,v_{q+1})$ 
with form 
$$(\overbrace{u,\dots,u}^j,\overbrace{v,\dots,v}^{q+1-j})$$
where $u$ and $v$ are adjacent vertices in $G$ (and $0\leq j\leq q+1$),
where tuples  of the form
$$(\overbrace{u,\dots,u}^j,\overbrace{v,\dots,v}^{q+1-j})\text{ and }
(\overbrace{v,\dots,v}^{q+1-j},\overbrace{u,\dots,u}^j)$$
are identified;
\item edges of  $\mathsf I(G)$ are those pairs $\{x,y\}$ where $x$ and $y$ have
representative of the form
$$(\overbrace{u,\dots,u}^j,\overbrace{v,\dots,v}^{q+1-j})\text{ and }
(\overbrace{u,\dots,u}^{j+1},\overbrace{v,\dots,v}^{q-j})$$
(for some $u,v\in G$ and $0\leq j\leq q$).
\end{itemize}

The main interest of such a logical construction (called {\em interpretation}) lies in the following property:

\begin{proposition}[See, for instance \cite{Lascar2009}, p.~180]
\label{prop:interpret}
For every first-order formula $\phi[v_1,\dots,v_k]$
  there exists a formula $\mathsf I(\phi)[\overline{w}_1,\dots,\overline{w}_k]$  with $k(q+1)$ free variables (each $\overline{w}_i$ represents a succession of $(q+1)$ free variables) such that for every 
graph $G$ and every  
$(x_1,\dots,x_k)\in \mathsf I(G)^k$  the three following conditions are equivalent:
\begin{enumerate}
	\item $\mathsf I(G)\vDash \phi[x_1,\dots,x_k]$;
	\item there exist $\overline{b}_1\in x_1,\dots,\overline{b}_k\in x_k$ such that $G\vDash \mathsf I(\phi)[\overline{b}_1,\dots,\overline{b}_k]$;
	\item for all $\overline{b}_1\in x_1,\dots,\overline{b}_k\in x_k$ it holds $G\vDash \mathsf I(\phi)[\overline{b}_1,\dots,\overline{b}_k]$.
\end{enumerate}
\end{proposition}

In particular, it holds:
\begin{corollary}
\label{cor:sub}
For every sentence (i.e. closed first order formula) $\phi$ (in the language of graphs)
there exists a sentence $\psi$ such that for every graph $G$ we have
\begin{equation}
G\vDash \psi\quad\iff\quad {\rm Sub}_{2p}(G)\vDash\phi.
\end{equation}
\end{corollary}

\begin{lemma}
If the homomorphism preservation theorem holds for a hereditary class of graphs $\mathcal C$, it also holds for the class ${\rm Sub}_{q}(\mathcal C)$ 
of all $q$-subdivisions of the graphs in $\mathcal C$.
\end{lemma}  
\begin{proof}
If $q$ is odd then the property is obvious as $\mathcal C$ contains at most two homomorphism equivalence classes, the one of $K_1$ and the one of $K_2$. Hence we can assume $q$ is even and we define $p=q/2$.

Let $\phi$ be a sentence preserved by homomorphisms on ${\rm Sub}_{2p}(\mathcal C)$, where $\mathcal C$ is a hereditary class of graphs on
which the homomorphism preservation theorem holds.
Then we shall prove that there exists a finite family of $2p$-subdivided graphs $\mathcal F$, all of which satisfy $\phi$, and such that  for any graph $G$ it holds
\begin{equation}
{\rm Sub}_{2p}(G)\vDash\phi\quad\iff\quad\exists F\in\mathcal F\quad {\rm Sub}_{2p}(F)\rightarrow {\rm Sub}_{2p}(G).
\end{equation}

According to Corollary~\ref{cor:sub} there exists a sentence $\psi$ such that for every graph $G$ it holds
$$
G\vDash\psi\quad\iff\quad{\rm Sub}_{2p}(G)\vDash \phi.
$$
Assume that $G\vDash \psi$ and $G\rightarrow H$, with $G,H\in\mathcal C$. Then
${\rm Sub}_{2p}(G)\vDash \phi$ and ${\rm Sub}_{2p}(G)\rightarrow {\rm Sub}_{2p}(H)$.
As $\phi$ is preserved by homomorphisms on ${\rm Sub}_{2p}(\mathcal C)$ we get
${\rm Sub}_{2p}(H)\vDash \phi$ hence $H\vDash\psi$. Thus $\psi$ is preserved by homomorphisms on $\mathcal C$. 
As the homomorphism preservation theorem holds by assumption on $\mathcal C$,
$\psi$ is equivalent on $\mathcal C$ with a positive first-order formula, that is: there exits a finite family $\mathcal F_0$ of finite graphs such that
for every $G\in\mathcal C$ it holds:
$$
G\vDash\psi\quad\iff\quad\exists F\in\mathcal F_0\quad F\rightarrow G.
$$
Moreover, by considering the subgraphs induced by the homomorphic images of the graphs $F\in\mathcal F_0$ and as $\mathcal C$ is hereditary, we can assume $\mathcal F_0\subseteq\mathcal C$.
Thus every $F\in\mathcal F_0$ satisfies $\psi$ hence the $2p$-subdivision of the graphs in $\mathcal F_0$ satisfy $\phi$.
Let $\mathcal F$ be the set of the $2p$-subdivisions of the graphs in $\mathcal F_0$.
As $\phi$ is preserved by homomorphisms on ${\rm Sub}_{2p}(\mathcal C)$ it follows that for every graph $G\in\mathcal C$ if there exists $F\in\mathcal F$ such that
$F\rightarrow{\rm Sub}_{2p}(G)$ then ${\rm Sub}_{2p}(G)$ satisfies $\phi$.
Conversely, if ${\rm Sub}_{2p}(G)$ satisfies $\phi$ for some $G\in\mathcal C$ then $G$ satisfies $\psi$, thus there exists $F\in\mathcal F_0$ such that $F\rightarrow G$ hence
${\rm Sub}_{2p}(F)\rightarrow {\rm Sub}_{2p}(G)$.
\end{proof}

We deduce this extension of Rossman's theorem to the class of $p$-subdivided graphs:
\begin{corollary}
\label{cor:subhpt}
For every integer $p$, the homomorphism preservation theorem holds for ${\rm Sub}_{p}(\GRA)$.
\end{corollary}

For a discussion on relativization of the homomorphism preservation theorem, we refer the reader
 to \cite[Chapter~10]{Sparsity}.

We summarize below the results obtained in this section:
\begin{lemma}
\label{lem:chpt}
Let $\mathcal C$ be a hereditary class of graphs. Assume $H$-coloring is first-order definable on $\mathcal{C}$.
\begin{itemize}
\item If $\mathcal{C}$ is topologically closed and somewhere dense, then there
exist an integer $p$ (independent of $H$) and a finite set $\mathcal{F}$ of finite graphs such that for every graph $G$ it holds
\begin{equation}
\label{eq:hpt1}
{\rm Sub}_{2p}(G)\nrightarrow H\quad\iff\quad\exists F\in\mathcal F:\ 
F\rightarrow {\rm Sub}_{2p}(G).
\end{equation}
\item If $\mathcal{C}$ is addable and nowhere dense, then there
exists a finite set $\mathcal{F}$ of finite graphs such that for every graph $G\in\mathcal{C}$ it holds
\begin{equation}
\label{eq:hpt2}
G\nrightarrow H\quad\iff\quad\exists F\in\mathcal F:\  F\rightarrow G.
\end{equation}
\end{itemize}
\end{lemma}
\begin{proof}
If $\mathcal{C}$ is topologically closed and somewhere dense, then there
exist an integer $p$ such that the class $\mathcal S_{2p}$ of all $2p$-subdivisions of finite graphs is a subclass of $\mathcal C$. As $H$-coloring is  first-order definable on $\mathcal C$ (thus on $\mathcal S_{2p}$) and as
the homomorphism preservation theorem holds for $\mathcal S_{2p}$ 
(according to Corollary~\ref{cor:subhpt}), $H$-coloring may be expressed on 
 $\mathcal S_{2p}$ by an existential first-order formulas, that is there exists a finite set $\mathcal{F}$ of finite graphs such that for every graph $G$ equivalence~\eqref{eq:hpt1} holds.

If $\mathcal{C}$ is addable and nowhere dense, then the existence of
 a finite set $\mathcal{F}$ of finite graphs such that for every graph $G\in\mathcal{C}$ equivalence~\eqref{eq:hpt2} holds immediately follows from Theorem~\ref{thm:hptnd}.
\end{proof}

\section{Connectivity of Forbidden Graphs}
Homomorphism preservation theorems allow to reduce the study of first-order colorings of a class $\mathcal C$ to the study of finite restricted dualities of $\mathcal C$, that is of pairs $(\mathcal F,H)$, where $\mathcal F$ is a finite set of finite graphs, where $H$  is a graph, and where it holds
\begin{equation}
\forall G\in\mathcal C:\qquad (\exists F\in\mathcal F:\ F\nrightarrow G)\iff(G\rightarrow H).
\end{equation}

In general, it is not required that $F\nrightarrow H$ when $F\in\mathcal F$, nor is it required that the graphs in $\mathcal F$ are connected. We shall see that 
if the class $\mathcal C$ is addable or monotone then we can require the graphs in $\mathcal F$ to be connected, and that if $\mathcal C$ is monotone we can further require every $F\in\mathcal F$ belongs to $\mathcal C$ (hence cannot be homomorphic to $H$).

For a graph $G$, we define ${\rm Pre}(G)$ has the set of all the graphs $G'$ 
obtained from $G$ by identifying two vertices of $G$.
Note that we have the following property:
\begin{lemma}
\label{lem:pre}
Let $F,G$ be graphs.
Then $F\rightarrow G$ if and only if
 either $F$ is isomorphic to a subgraph of $G$, or there is $F'\in{\rm Pre}(F)$ such that $F'\rightarrow G$.
\end{lemma}
\begin{proof}
Assume $f: F\rightarrow G$ is a homomorphism. Either $f$ is injective and 
$F$ is isomorphic to a subgraph of $G$, or at least two vertices of $F$ are identified by $f$ and thus there is $F'\in{\rm Pre}(F)$ such that $F'\rightarrow G$.
Conversely, if $F$ is isomorphic to a subgraph of $G$ or if there is $F'\in{\rm Pre}(F)$ such that $F'\rightarrow G$ then obviously $F\rightarrow G$.
\end{proof}

Let $\mathcal C$ be a class of graphs and let $(\mathcal{F},H)$ be a restricted duality of $\mathcal C$.
We say that the set $\mathcal F$ is {\em minimal} if 
\begin{itemize}
\item every element of $\mathcal{F}$ is a {\em core} (that is a graph $F$ such that every homomorphism $F\rightarrow F$ is an automorphism);
\item for any proper subset of $\mathcal{F}'$ of $\mathcal F$,
the pair $(\mathcal{F}',H)$ is not a restricted duality of $\mathcal C$,
\item and for any $F\in\mathcal F$, the pair $(\mathcal F\setminus\{F\}\cup{\rm Pre}(F),H)$ is not a restricted duality of $\mathcal C$.
\end{itemize}

It is clear that we can restrict our attention to minimal restricted dualities, as
if $(\mathcal{F},H)$ is a restricted duality of a class $\mathcal C$ then there exists minimal $\mathcal F'$ such that $(\mathcal{F}',H)$ is a restricted duality of  $\mathcal C$.

\begin{lemma}
Let $(\mathcal F,H)$ be a restricted duality of $\mathcal C$, with $\mathcal F$ minimal. 
\begin{itemize}
\item If the class $\mathcal C$ is addable, then every graph in $\mathcal F$ is connected.
\item If the class $\mathcal C$ is monotone, then every graph in $\mathcal F$ is connected and $\mathcal F\subseteq \mathcal C$ (hence $F\nrightarrow H$ holds for every $F\in\mathcal{F}$).
\end{itemize}
\end{lemma}
\begin{proof}

Assume $\mathcal C$ is addable, and assume for contradiction that 
$F_1+F_2\in\mathcal F$. By minimality of $\mathcal F$, neither 
$(\mathcal F\setminus\{F_1+F_2\}\cup\{F_1\},H)$ nor
$(\mathcal F\setminus\{F_1+F_2\}\cup\{F_2\},H)$ are restricted dualities of $\mathcal C$. Hence there exist $G_1,G_2\in\mathcal C$ such that
$F_1\nrightarrow G_1$, $F_2\rightarrow G_1$ (hence $G_1\rightarrow H$), $F_1\rightarrow G_2$, $F_2\nrightarrow G_2$ (hence $G_2\rightarrow H$). 
As $\mathcal{C}$ is addable, $G_1+G_2\in\mathcal C$. But
$F_1+F_2\rightarrow G_1+G_2$, what contradicts $G_1+G_2\rightarrow H$.

Assume $\mathcal C$ is monotone, and assume 
$F\in\mathcal F$. By minimality of $\mathcal F$, there exists $G\in\mathcal C$  such that $F\rightarrow\mathcal G$ but no graph
$F'\in{\rm Pre}(F)$ is homomorphic to $G$. Thus, according to 
Lemma~\ref{lem:pre}, $F$ is isomorphic to a subgraph of $G$ hence, as $\mathcal C$ is monotone, $F\in\mathcal C$. Assume for contradiction that
$F=F_1+F_2$. Then $F_1,F_2\in\mathcal C$ and none of $F_1,F_2$ is homomorphic to the other. By minimality of $\mathcal F$ it follows that
$F_1\rightarrow H$, $F_2\rightarrow H$ hence $F_1+F_2\rightarrow H$, contradicting $F_1+F_2\in\mathcal F$.
\end{proof}
\section{Restricted Dualities}
\label{sec:ard}
As restricted dualities appear as the central notion when dealing with first-order definable coloring, we take time to define and characterize classes with all restricted dualities in the more general framework of relational structures.

\subsection{Classes of Relational Structures}
\label{sec:rel}
We recall some basic definitions, notations and result of model theory. Our terminology is standard, cf \cite{Ebbinghaus1996,Lascar2009}:

A {\em signature} $\sigma$ is a finite set of relation symbols,
each with a specified arity. A {\em $\sigma$-structure} $\mathbf{A}$ 
consists of a {\em universe} $A$, or {\em domain}, and an {\em interpretation} which associates to each relation symbol $R \in\sigma$ of some arity $r$, a
relation $R^{\mathbf A} \subseteq A^r$.

A $\sigma$-structure $\mathbf{B}$ is a {\em substructure} of $\mathbf{A}$ if $B\subseteq A$ and $R^{\mathbf B} \subseteq R^{\mathbf A}$ for every $R\in\sigma$. It is an {\em induced substructure} if $R^{\mathbf B} = R^{\mathbf A} \cap B^r$ for every $R\in\sigma$ of arity $r$. Notice the analogy with the graph-theoretical concept of subgraph and induced subgraph.
A substructure $\mathbf{B}$ of $\mathbf{A}$ is {\em proper} if $\mathbf{A}\neq\mathbf{B}$. If $\mathbf{A}$
 is an induced substructure
of $\mathbf{B}$, we say that $\mathbf{B}$ is an {\em extension} of $\mathbf{A}$.
If $\mathbf{A}$ is a proper induced substructure, then
$\mathbf{B}$ is a {\em proper extension}. If $\mathbf{B}$ is the disjoint union of $\mathbf{A}$ with another $\sigma$-structure, we
say that $\mathbf{B}$ is a {\em disjoint extension} of $\mathbf{A}$. If $S\subseteq A$ is a subset of the universe of $\mathbf{A}$,
then $\mathbf{A}\cap S$ denotes the {\em induced substructure generated by} $S$; in other words, the universe
of $\mathbf{A}\cap S$ is $S$, and the interpretation in $\mathbf{A} \cap S$ of the $r$-ary relation symbol $R$ is
$R^{\mathbf A} \cap S^r$.

The {\em Gaifman graph} ${\rm Gaifman}(\mathbf{A})$ of a $\sigma$-structure $\mathbf A$ is the graph with vertex set $A$ in which two vertices $x\not\simeq y$ are adjacent if and only if there exists a relation $R$ of arity $k\geq 2$ in $\sigma$ and $v_1,\dots,v_k\in A$ such that $\{x,y\}\subseteq\{v_1,\dots,v_k\}$ and $(v_1,\dots,v_k)\in R^{\mathbf A}$.

A {\em block} of a $\sigma$-structure $\mathbf A$ is a tuple
$(R,x_1,\dots,x_k)$ such that $R\in\sigma$ has arity $k$ and $(x_1,\dots,x_k)\in R^{\mathbf A}$.
The {\em incidence graph} ${\rm Inc}(\mathbf A)$ is the bipartite graph $(A,B,E)$ where $A$ is the universe of $\mathbf A$, $B$ is the set of all {\em blocks} of $\mathbf A$, and $E$ is the set of the pairs $\{(R,x_1,\dots,x_k),y\}\subseteq B\times A$ such that $y\in\{x_1,\dots,x_k\}$.
Thus for us  ${\rm Inc}(\mathbf A)$ is a simple graph. No multiple edges are needed for our purposes.


A {\em homomorphism} $\mathbf{A}\rightarrow \mathbf{B}$ between two $\sigma$-structure is defined as a mapping $f:A\rightarrow B$ which satisfies for every relational symbol $R\in\sigma$ the following:
\begin{equation*}
(x_1,\dots,x_k)\in R^{\mathbf A}\quad\Longrightarrow\quad (f(x_1),\dots,f(x_k))\in R^{\mathbf B}.
\end{equation*}
The class of all $\sigma$-structures is denoted by ${\rm Rel}(\sigma)$.

The definition of bounded expansion extends to classes of relational structures: 
a class $\mathcal C$ of relational structures has {\em bounded expansion} if
the class of the Gaifman graphs of the structures in $\mathcal{C}$ has bounded expansion.
 It is immediate that two relational
structures have the same Gaifman graph if they have the same incidence graph,
but that the converse does not hold in general. For a class of relational
structures $\mathcal C$, denote by ${\rm Inc}(\mathcal C)$ the class of all the
incidence graphs ${\rm Inc}(\mathbf{A})$ of the relational structures
${\mathbf A}\in\mathcal C$. 

\begin{proposition}[\cite{Sparsity}]
\label{prop:IG}
Assume that the arities of the relational symbols in $\sigma$ are bounded, and
let $\mathcal C$ be an infinite class of $\sigma$-structures. Then
the class $\mathcal C$ has bounded expansion if and only if the class ${\rm Inc}(\mathcal C)$ has bounded expansion.
\end{proposition}

\subsection{Classes with all restricted dualities}
A class of $\sigma$-structures $\mathbf A$ has {\em all restricted dualities} if
every non-empty connected $\sigma$-structure has a restricted dual for $\mathcal
C$, that is: for every non-empty connected $\sigma$-structure $\mathbf F$ there
exists a $\sigma$-structure $\mathbf D$ such that $\mathbf F\nrightarrow \mathbf D$ and

$$\forall \mathbf A\in\mathcal C:\qquad(\mathbf F\rightarrow \mathbf
A)\quad\iff\quad (\mathbf A\nrightarrow \mathbf D).$$

Note that this definition implies that also for any finite set $\mathbf  F_1,\mathbf  F_2,\dots, \mathbf  F_t$ 
 of connected $\sigma$-structures  there exists a $\sigma$-structure $\mathbf  D$ such that
$\mathbf F_i\nrightarrow \mathbf D$ (for $1\leq i\leq t$) and

$$\forall \mathbf A\in\mathcal C:\qquad(\exists i\leq t:\ \mathbf F_i\rightarrow \mathbf
A)\quad\iff\quad (\mathbf A\nrightarrow \mathbf D).$$

For a structure ${\mathbf A}$ and an integer $t$, define
$\Theta^t({\mathbf A})$ as the minimum order of a structure 
$\mathbf B$ such that
\begin{itemize}
  \item $\mathbf A\rightarrow \mathbf B$,
  \item every substructure $\mathbf F$ of $\mathbf B$ with order $|F|<t$ has a homomorphism to  $\mathbf A$.
\end{itemize}
Intuitively, such a structure $\mathbf B$ can be seen as 
 approximate core of $\mathbf{A}$: For $t \geq |B|$, $\mathbf{A}$  and $\mathbf{B}$ are 
 homomorphism-equivalent and $\mathbf{B}$ is the core of $\mathbf{A}$ 
 (alternately, $\mathbf{B}$ is the minimal retract of $\mathbf{A}$).
 A structure $\mathbf B$ with the above properties and order
$\Theta^t(\mathbf A)$ is called a {\em $t$-approximation} of (the
homomorphism equivalence class of) $\mathbf A$.

It appears that existence of a uniform approximation is equivalent for a class to having all restricted dualities. This is formalized by Theorem~\ref{thm:approx},  stated in the introduction. Theorem~\ref{thm:approx} will be proved now:

\begin{proof}[Proof of Theorem~\ref{thm:approx}]
Assume $\mathcal C$ is bounded and has all restricted dualities and let
$t\in\bbbn$ be an integer. 
Let $\mathbf Z$ be a strict bound of $\mathcal C$, that is a structure such that
for every $\mathbf A\in\mathcal C$ it holds $\mathbf A\rightarrow \mathbf Z$ but
$\mathbf Z\nrightarrow\mathbf A$.
As the sequence $(\Theta^t({\mathbf A}))_{t\in\bbbn}$ is obviously
non-decreasing, we may assume without loss of generality that $t\geq |Z|$. For a structure
${\mathbf A}\in\mathcal C$, let $\mathcal F_t({\mathbf A})$ be the set of all
connected cores ${\mathbf T}$ of order at most $t$ such that ${\mathbf
T}\nrightarrow {\mathbf A}$. This set is not empty as it contains the core of
$\mathbf Z$. For ${\mathbf T}\in\mathcal F_t({\mathbf A})$, let
${\mathbf D}_{\mathbf T}$ be the dual of ${\mathbf T}$ relative to $\mathcal C$
and let ${\mathbf A}'$ be the product of all the ${\mathbf D}_{\mathbf T}$ for
${\mathbf T}\in\mathcal F_t({\mathbf A})$.
First notice that for every ${\mathbf T}\in\mathcal F_t({\mathbf A})$ we have
${\mathbf T}\nrightarrow {\mathbf A}$ hence ${\mathbf A}\rightarrow {\mathbf
D}_{\mathbf T}$. It follows that ${\mathbf A}\rightarrow {\mathbf A}'$. Let
${\mathbf T}'$ be a connected substructure of order at most $t$ of $\mathbf A'$.
Suppose for a contradiction that ${\mathbf T}'\nrightarrow {\mathbf
A}$. Then ${\rm Core}({\mathbf T}')\in\mathcal F_t({\mathbf A})$ hence
${\mathbf A}'\rightarrow {\mathbf D}_{{\mathbf T}'}$ thus ${\mathbf
T}'\nrightarrow {\mathbf A}'$ (as for otherwise ${\mathbf T}'\rightarrow
{\mathbf D}_{{\mathbf T}'}$), a contradiction. Thus ${\mathbf T}'\rightarrow {\mathbf
A}$. It follows that $\Theta^t({\mathbf A})\leq |A'|\leq
C(t)$ for some suitable finite constant $C(t)$ independent of ${\mathbf A}$
(for instance, one can choose $C(t)$ to be the product of the orders of
all the duals relative to $\mathcal C$ of connected cores of order at most $t$).

Conversely, assume that 
we have $\sup_{{\mathbf A}\in\mathcal C}\Theta^t({\mathbf
A})<\infty$ for every $t\in\bbbn$.
The class $\mathcal C$ is obviously
bounded by the disjoint union of all non-isomorphic minimal order
$1$-approximations of the structures in $\mathcal C$.
Let ${\mathbf F}$ be a connected
$\sigma$-structure, let $t\geq |F|$, and let $\mathcal D$ be a set
of $t$-approximations of all the structures ${\mathbf A}\in\mathcal C$ such that
${\mathbf F}\nrightarrow {\mathbf A}$. As all the
$\Theta^t({\mathbf A})$ are bounded by some constant $C(t)$,
the set $\mathcal D$ is finite. 
 If $\mathcal D$ is empty, let $D_t(\mathbf F)$ be the empty substructure.
Otherwise, let $D_t({\mathbf F})$ be the disjoint union of all the graphs in
$\mathcal D$.
First notice that ${\mathbf F}\nrightarrow D_t({\mathbf F})$ as for otherwise
${\mathbf F}$ would have a homomorphism to some structure in $\mathcal D$ (as
${\mathbf F}$ is connected), that is  to some $t$-approximation $\mathbf B'$ of
a structure ${\mathbf B}$ such that ${\mathbf F}\nrightarrow {\mathbf B}$ (this
 would contradict $\mathbf F\rightarrow {\mathbf B'}$). Also, if ${\mathbf
F}\rightarrow {\mathbf A}$ then ${\mathbf A}\nrightarrow D_t({\mathbf F})$ 
(for otherwise ${\mathbf F}\rightarrow D_t({\mathbf F})$) and if ${\mathbf
F}\nrightarrow {\mathbf A}$ then $\mathcal D$ contains a $t$-approximation
$\mathbf A'$ of $\mathbf A$ thus ${\mathbf A}\rightarrow
D_t({\mathbf F})$. Altogether, $D_t({\mathbf F})$ is a dual of ${\mathbf F}$
relative to $\mathcal C$.
\end{proof}

We proved in \cite{POMNIII} that bounded expansion classes have all restricted dualities:
\begin{theorem}
\label{thm:dual}
Let $\mathcal C$ be a class with bounded expansion. Then for every connected graph $F$ there exists a graph $D$ such that $(F,D)$ is a restricted homomorphism duality for $\mathcal C$:
\begin{equation}
\forall G\in\mathcal C\qquad (F\rightarrow G)\quad\iff\quad(G\nrightarrow H).
\end{equation}
\end{theorem}

Theorem~\ref{thm:dual} naturally extends to relational structures by considering Gaifman graphs. We sketch a proof of this generalization, which is based on the above Theorem~\ref{thm:approx}.

\begin{theorem}
\label{thm:dual_struc}
Let $\mathcal K$ be a class  of relational structures.
If the class of the Gaifman graphs of the structures in $\mathcal K$ has bounded expansion then the class $\mathcal K$ has all restricted dualities.
\end{theorem}
\begin{proof}[Sketch of the proof]
Let $\mathcal K$ be a class  of relational structures.
Assume the class of the Gaifman graphs of the structures in $\mathcal K$ has bounded expansion.
Let $\mathbf A\in\mathcal{K}$, and let $t\in\bbbn$ be at least as large as the maximum arity
of a relation in the signature of $\mathbf{A}$.

The {\em tree-depth} ${\rm td}(G)$ of a graph $G$ is the minimum height of a rooted forest whose closure includes $G$ as a subgraph. One of the most interesting properties of  tree-depth is that
there exists a function $\digamma:\bbbn\rightarrow\bbbn$ with the property
that if the Gaifman graph of a structure $\mathbf{B}$ has tree-depth at most $t$ then
there exists a homomorphism $f:\mathbf{B}\rightarrow\mathbf{B}$ such that 
$|f(\mathbf B)|\leq\digamma(t)$~\cite{Taxi_tdepth}.
For integer $t$, we defined in \cite{Taxi_tdepth} the graph invariant $\chi_t$ as follows:
for a graph $G$, $\chi_t(G)$ is the minimum number of colors needed in a coloring of $G$ such that
the union of every subset of $k\leq t$ color classes induces a subgraph with tree-depth at most $k$ (such colorings are called {\em low tree-depth colorings}).
It has been proved in~\cite{POMNI} that a class of graphs $\mathcal C$ has bounded expansion if and only if for every integer $t$ it holds $\sup\{\chi_t(G):\ G\in\mathcal{C}\}<\infty$
 (this is related to Theorem~\ref{thm:BEchi} above).

Consider a coloring $c$ of the Gaifman graph of $\mathbf{A}$ by $N=\chi_t({\rm Gaifman}(\mathbf A))$ colors, which is such that
the union of every subset of $k\leq t$ color classes induces a subgraph with tree-depth at most $k$.  It follows that 
for each $I\in\binom{[N]}{t}$ there exists a homomorphism
$f_I:\mathbf A_I\rightarrow \mathbf A_I$ such that $|f_I(\mathbf A_I)|\leq\digamma(t)$, where
$\mathbf{A}_I$ denotes the substructure of $\mathbf{A}$ induced by elements with color in $I$.
Define the equivalence relation $\sim$ on the domain of $\mathbf A$ by
 $$x\sim y\quad\iff\quad c(x)=c(y)\text{ and }\forall I\in\binom{[N]}{t}\ f_I(x)=f_I(y).$$ 
Define the structure $\hat{\mathbf{A}}$ (with same signature as $\mathbf A$) whose domain is the set of the
equivalence classes $[x]\in  A/\sim$, and relations are defined by
$$([x_1],\dots,[x_{k_i}])\in R_i^{\hat{\mathbf{A}}}\quad\iff\quad
\forall I\in\binom{[N]}{t}\  (f_I(x_1),\dots,f_I(x_{k_i}))\in R_i^{\mathbf{A}}.$$

 We also define a
$N$-coloration of $\hat{\mathbf{A}}$ by $\hat c([x])=c(x)$. One checks
easily that $\hat{\mathbf{A}}$ and $\hat
c$ are well defined.
By construction, $x\mapsto [x]$ is a homomorphism $\mathbf A\rightarrow \hat{\mathbf{A}}$.
Moreover,  for every $I\in\binom{[N]}{t}$ the mapping 
$[x]\mapsto f_I(x)$ is a homomorphism $\hat{\mathbf{A}}_I\rightarrow \mathbf{A}_I$ 
(where $\hat{\mathbf{A}}_I$ is the substructure of $\hat{\mathbf{A}}$ induced by colors in $I$).
It follows that 
$$|\Theta^{t}(\mathbf{A})|\leq|\hat{A}|\leq \digamma(t)^{N^t}
\leq \digamma(t)^{\chi_t({\rm Gaifman}(\mathcal K))^t}.$$
According to Theorem~\ref{thm:approx}, this implies that the class $\mathcal{K}$ has
all restricted dualities.
\end{proof}
For an alternate proof of this Theorem, we refer the reader to~\cite{POMNIII, Sparsity}.
\subsection{Topologically closed classes of graphs with all restricted dualities}
\label{sec:sub}
The special case of topologically closed classes of graphs is of particular interest here, and we have in this case a much simpler characterization of the classes that have all restricted dualities. We are now ready to prove Theorem~\ref{thm:sub}.

\begin{proof}[Proof of Theorem~\ref{thm:sub}]
The proof follows from the next three implications:
\begin{itemize}
\item (1)$\ \Rightarrow$ (2) is a direct consequence of Theorem~\ref{thm:dual}.
\item (2)$\ \Rightarrow$ (3) is straightforward (consider for $H_g$ a dual of $C_g$ for $\mathcal C$).
	\item (3)$\ \Rightarrow$ (1) is proved by contradiction: assume that (3) holds and that $\mathcal C$ does not have bounded expansion. According to Theorem~\ref{thm:BEchi} there exists an integer $p$ such that $\mathcal C\shtm p$ has unbounded chromatic number. As $\mathcal C$ is topologically closed there exists an odd integer $g\geq p$ and a graph $G_0\in\mathcal C$ such that
	$G_0$ is the $(g-1)$-subdivision of a graph $H_0$ with chromatic number $\chi(H_0)>|H_g|$. According to (3), there exists a homomorphism $f:G_0\rightarrow H_g$. As $C_g\nrightarrow H_g$, the ends of
	a path of length $g$ cannot have the same image by $f$. It follows that any two adjacent vertices in $H_0$ correspond to branching vertices of $G_0$ which are mapped by $f$ to distinct vertices of $H_g$. It follows that $\chi(H_0)\leq |H_g|$, a contradiction.
\end{itemize}
\end{proof}

\section{On First-Order Definable $H$-colorings}
\label{sec:focol}
In this section we prove our main characterization result on first-order definable colorings, stated in the introduction as Theorem~\ref{thm:main}.

\begin{proof}[Proof of Theorem~\ref{thm:main}]
Assume that the class $\mathcal C$ is somewhere dense.
As $\mathcal{C}$ is  topologically closed, there exists 
an integer $p$ such that ${\rm Sub}_{2p}(\GRA)\subseteq\mathcal C$.

Assume for contradiction that there exists a non-bipartite graph $H$
(different from $K_1$) of odd-girth strictly greater than $2p+1$
and a first-order formula $\Phi$ such that 
for every graph $G\in\mathcal C$ holds
$$(G\vDash \Phi)\quad\iff\quad (G\rightarrow H).$$
According to Corollary~\ref{cor:subhpt}, as $\neg\Phi$ is preserved by homomorphisms on $\mathcal C$ (hence on ${\rm Sub}_{2p}(\GRA)$) it is equivalent
on ${\rm Sub}_{2p}(\GRA)$ with an existential first-order formula, that is: there exists a finite family $\mathcal F$ such that
for every graph $G$ it holds:
$$
\forall F\in\mathcal F\ F\nrightarrow  {\rm Sub}_{2p}(G)\quad\iff\quad {\rm Sub}_{2p}(G)\rightarrow H.
$$
Clearly, the graphs in $\mathcal F$ are non-bipartite. Let $g$ be the maximum of girth of graphs in $\mathcal F$ and let $G$ be a graph
with chromatic number $\chi(G)>|H|$ and odd-girth ${\rm odd-girth}(G)>g$. Then 
for every $F\in\mathcal F$ we have $F\nrightarrow  {\rm Sub}_{2p}(G)$ hence ${\rm Sub}_{2p}(G)\rightarrow H$. However, has the odd-girth of $H$ is strictly greater than $2p+1$ two 
branching vertices of ${\rm Sub}_{2p}(G)$ corresponding to adjacent vertices of $G$ cannot be mapped to a same vertex. It follows that $|H|\geq\chi(G)$, a contradiction.

To the opposite, if  $\mathcal C$ has bounded expansion, there exists for every integer $p$ a non-bipartite graph $H_p$ of odd-girth strictly greater than $2p+1$ and
a first order formula $\Phi_p$ such that for every graph $G\in\mathcal C$ holds
$$(G\vDash \Phi_p)\quad\iff\quad (G\rightarrow H_p).$$
Indeed, consider for $\Phi_p$ the formula asserting that $G$ contains an odd cycle of length at most
$2p+1$, and for $H_p$ the restricted dual of the cycle $C_{2p+1}$ with respect to $\mathcal C$ (whose existence follows from Theorem~\ref{thm:dual}.

Now assume that Conjecture~\ref{conj:nd} holds, and that $\mathcal C$ is a hereditary topologically closed addable nowhere dense class that is not a bounded expansion class.
Then there exists an integer $p$ such that 
$\mathcal C$ includes $2p$-subdivisions of graph with arbitrarily large chromatic number and girth.
 Assume for contradiction that there is a non-bipartite graph $H_p$ of odd-girth strictly greater than $2p+1$ such that $H_p$-coloring is first-order definable on $\mathcal C$, and let $\Phi$ be a formula such that for every
$G\in\mathcal C$ it holds $G\rightarrow H_p$ if and only if $G\models\Phi$.
According to Theorem~\ref{thm:hptnd}, there exists a finite family $\mathcal F$ of finite graphs such that $G\models H_p$ if and only if no graph in $\mathcal F$ is homomorphic to $G$. The graphs in $\mathcal F$ are non-bipartite ($\mathcal C$ contains long odd cycles homomorphic to $H_p$). Let $g$ be the maximum of the odd-girths of the graphs in $\mathcal{F}$.
Let $G\in\mathcal C$ be a $2p$-subdivision of a graph $H$ with girth greater
than $g$ and chromatic number $\chi(H)>|H_p|$. As the girth of $G$ is greater than $g$, no graph in $\mathcal{F}$ is homomorphic to $G$ hence there exists a homomorphism $f:G\rightarrow H_p$. As $C_{2p+1}\nrightarrow G$ it follows that vertices of $G$ linked by a path of length $2p+1$ are mapped to distinct vertices of $H_p$ hence $f$ defines a homomorphism $H\rightarrow K_{|H_p|}$, contradicting  $\chi(H)>|H_p|$.
\end{proof}

  \section*{References}
\providecommand{\noopsort}[1]{}\providecommand{\noopsort}[1]{}
\providecommand{\bysame}{\leavevmode\hbox to3em{\hrulefill}\thinspace}
\providecommand{\MR}{\relax\ifhmode\unskip\space\fi MR }
\providecommand{\MRhref}[2]{%
  \href{http://www.ams.org/mathscinet-getitem?mr=#1}{#2}
}
\providecommand{\href}[2]{#2}


\end{document}